\documentclass[11pt,reqno,a4paper]{amsart}

\usepackage{amsmath}
\usepackage{amssymb}
\usepackage{graphicx}
\usepackage{mathtools}
\usepackage{enumerate}

\newtheorem{thm}{Theorem}
\newtheorem*{thm*}{Theorem}

\newtheorem*{lem*}{Lemma}

\theoremstyle{remark}
\newtheorem{rema}{Remark}[section]

\newcommand{\bna}{\mbox{\boldmath$ \nabla$}}
\newcommand{\blp}{\mbox{\boldmath$ \Delta$}}

\def\mb{\mathbf}

\def\mc{\mathcal}

\newlength{\equwidth}
\DeclareMathOperator{\Ric}{Ric}

\DeclareMathOperator{\id}{id}

\def\X{\mathfrak X}

\def\la{\lambda}

\def\Ups{\Upsilon}

\def\Om{\Omega}

\def\pa{\partial}

\def\goesto{\rightarrow}

\def\squig{\rightsquigarrow}

\def\na{\mathrm{\nabla}}

\def\rr{\ensuremath{\mathbb{R}}}

\def\O{\ensuremath{\mathrm{O}}}

\def\ce{\ensuremath{\mathcal{E}}}

\def\today{\ifcase\month\or
 January\or February\or March\or April\or May\or June\or
 July\or August\or September\or October\or November\or December\fi
 \space\number\day, \number\year}

\usepackage{amssymb}
\usepackage{paralist} 
\usepackage{graphicx} 
\usepackage{hyperref}
\usepackage[all]{xy}

\def\mb{\mathbf}

\def\mc{\mathcal}

\def\d{\operatorname{d}\!}

\usepackage{tensor}
\newcommand{\ind}{\indices}

\title[Ambient metrics of Patterson--Walker metrics]{Fefferman--Graham ambient metrics of Patterson--Walker metrics}

\author[Hammerl]{Matthias Hammerl}

\author[Sagerschnig]{Katja Sagerschnig}
\thanks{KS is an INdAM (Istituto Nazionale di Alta Matematica) research fellow} 
\author[\v{S}ilhan]{Josef \v{S}ilhan}
\thanks{J\v{S} was supported by the Czech science foundation (GA\v{C}R) under grant P201/12/G028} 

\author[Taghavi-Chabert]{Arman Taghavi-Chabert}

\author[\v{Z}\'adn\'ik]{Vojt\v{e}ch \v{Z}\'adn\'ik}

\address{M. H.: 
University of Greifswald, Department of Mathematics and Informatics, Walther-Rathenau-Str. 47, 17489 Greifswald, Germany
\newline K. S.: INdAM-Politecnico di Torino, Dipartimento di Scienze Matematiche, corso Duca degli Abruzzi 24, 10129 Torino, Italy \newline J. \v S. and A. T.-C..: 
Masaryk University, Faculty of Science, Kotl\'{a}\v{r}sk\'{a} 2, 61137 Brno, Czech Republic
\newline 
V. \v{Z}: Masaryk University, Faculty of Education, Po\v{r}\'\i\v{c}\'\i\ 31, 60300 Brno, Czech Republic
}
\email{matthias.hammerl@univie.ac.at, katja.sagerschnig@univie.ac.at, silhan@math.muni.cz, taghavia@math.muni.cz, zadnik@mail.muni.cz}
\date{\today}


\begin{document}

\maketitle

\begin{abstract}
\vspace*{-0.8cm}
Given an $n$-dimensional manifold $N$ with an affine connection $D$, we show that the associated Patterson--Walker metric $g$ on $T^*N$ admits a global and explicit Fefferman--Graham ambient metric. This provides a new and large class of conformal structures which are generically not conformally Einstein but for which the ambient metric exists to all orders and can be realized in a natural and explicit way. In particular, it follows that Patterson--Walker metrics have vanishing Fefferman--Graham obstruction tensors. As an application of the concrete ambient metric realization we show in addition that Patterson--Walker metrics have vanishing Q-curvature. 
\end{abstract}

\section{Introduction and main result}
Given a signature $(p,q)$ conformal structure $[g]$ on an $m=p+q$\ dimensional manifold $M$, it was shown in seminal work by Charles Fefferman and C. Robin Graham (see \cite{fefferman-graham,fefferman-graham-ambient}) that under specific conditions the conformal structure can be encoded equivalently as a signature $(p+1,q+1)$ pseudo-Riemannian metric $(\mb{M},\mb{g})$ with vanishing Ricci curvature. This description has been fundamental in constructing and classifying conformal invariants (see e.g. \cite{fefferman-graham, bailey-eastwood-graham}) and for constructing and studying conformally invariant differential operators (see \cite{graham-nonexistence, GJMS}). 

To build the Fefferman--Graham ambient metric for given local coordinates $x$ on $M$, one first considers the ray bundle of metrics in the conformal class $[g]$, written as $\rr_+\times \rr^m$ with coordinates $(t,x)$.
The ambient space $\mb{M}$ is obtained by adding a new transversal coordinate
$\rho\in\rr$, and then an \emph{ansatz} for the Fefferman-Graham ambient metric $\mb{g}$ is
\begin{align}\label{FG}
     \mb{g}=t^2g_{ij}(x,\rho)dx^i\odot dx^j+2\rho dt\odot dt+2tdt\odot d\rho,
\end{align}
where $g=g_{ij}(x,0)dx^idx^j$ is a representative metric in the conformal class.
It is directly visible from the formula that $\mb{g}$ is homogeneous of degree $2$ with respect to the \emph{Euler field} $t\pa_t$ on $\mb{M}$.

To show existence of a Fefferman-Graham ambient metric $\mb{g}$ for given $g$,  the \emph{ansatz}\ \eqref{FG} determines an iterative procedure to determine $g_{ij}(x,\rho)$ as a Taylor series in $\rho$ satisfying $\Ric(\mb{g})=0$ to infinite order at $\rho=0$. For $m$ odd existence (and a natural version of uniqueness) of $\mb{g}$ as an infinity-order series expansion in $\rho$ is guaranteed for general $g_{ij}(x)$. For $m=2n$ even, the existence of an infinity order jet for $g_{ij}(x,\rho)$ with $\Ric(\mb{g})=0$ asymptotically at $\rho=0$ is obstructed at order $n$. Existence of $g_{ij}(x,\rho)$ as an infinity order series expansion in $\rho$ with $\Ric(\mb{g})=0$ asymptotically at $\rho=0$ is then equivalent to vanishing of the Fefferman--Graham obstruction tensor $\mc{O}$, which is a conformal invariant. Existence of $\mb{g}$ for $m=2n$ even does not in general guarantee uniqueness.

Results which provide global Fefferman--Graham ambient metrics,
where $\mb{g}$ can then be constructed in a natural way from $g$ and satisfies $\Ric(\mb{g})$ globally and not just asymptotically at $\rho=0$ are rare, both in the odd- and even-dimensional situation. A special instance where global ambient metrics can at least be shown to exist occurs for $g$ real--analytic, and $m$ either being odd or $m$ even and with obstruction tensor $\mc{O}$ of $g$ vanishing.
The simplest case of geometric origin for which one has global ambient metrics consists of locally conformally flat structures $(M,[g])$, where $(\mb{M},\mb{g})$ exists and is unique up to diffeomorphisms (see \cite{fefferman-graham-ambient}, ch. 7).
Another well known geometric case 
are conformal structures $(M,[g])$ which contain an Einstein metric $g$:
If $\Ric(g)=2\la(m-1)g$, then $\mb{g}$ on $\rr_+\times M\times \rr$ can be written directly in terms of $g$ as
\begin{align}\label{ambient-einstein}
     \mb{g}=t^2(1+\la\rho)^2g+2\rho dt\odot dt+2tdt\odot d\rho.
\end{align}

 In work by Thomas Leistner and Pawel Nurowski it was shown that the so called \emph{pp-waves} admit global ambient metrics in the odd-dimensional case and under specific assumptions in the even-dimensional case, see \cite{leistner-nurowski-pp}. Concrete and explicit ambient metrics for specific examples of families of conformal structures induced by generic $2$-distributions on $5$-manifolds 
and generic $3$-distributions on $6$ manifolds have been constructed in \cite{nurowski-explicit, leistner-nurowski-g2ambient, willse-missing, anderson-leistner-nurowski}.

The present article expands the class of metrics for which canonical ambient metrics exist globally to \emph{Patterson--Walker metrics}: Given
an affine connection $D$ on an $n$-manifold $N$ with $n\geq 2$, which is supposed to be torsion--free and to preserve a volume form, the \emph{Patterson--Walker metric} $g$ is a natural split--signature $(n,n)$ metric on $T^*N$, see recent work \cite{hsstz-walker} for historical background on Patterson--Walker metrics, references and a modern treatment. Our main result is:

\begin{thm}\label{thm1}
Let $D$ be a torsion-free affine connection on $N$ which preserves a volume form. Denote local coordinates on $N$ by $x^A$ and the induced canonical fibre coordinates on $T^*N$ by $p_A$. Let $\Gamma \ind{_A^C_B}$ denote the Christoffel symbols of $D$. 
Let 
  \begin{align}
    g &=  2 \, \d x^A \odot \d p_A -2 \, \Gamma \ind{_A^C_B} \, p_C  \, \d x^A \odot \d x^B \, \label{formula-PW}
  \end{align}
be the Patterson--Walker metric induced on $T^*N$ by $D$.
Then, with $\Ric_{AB}$ the Ricci curvature of $D$,
  \begin{align}\label{ambient-local}
     &\mb{g}=2\rho dt\odot dt+2tdt\odot d\rho \\
&+t^2(2dx^A\odot dp_A-2p_C\Gamma \ind{_A^C_B} dx^A\odot dx^B+\frac{2\rho}{n-1}\Ric_{AB}dx^A\odot dx^B)\notag
  \end{align}
 is a globally Ricci-flat Fefferman--Graham ambient metric for the conformal class $[g]$.
\end{thm}
We note that as an immediate consequence of the existence of the ambient metric for $(M,[g])$, the conformally invariant Fefferman-Graham obstruction tensor $\mc{O}$ associated to $[g]$ vanishes. 

It is not difficult to check Ricci-flatness of \eqref{ambient-local} directly: Specifically, one employs formula (3.17) of \cite{fefferman-graham}, which is applicable to any ambient metric in normal form \eqref{FG}. The computation is then based on the following key facts: The Ricci curvature of the Patterson--Walker metric $g$ is up to a constant multiple just the pullback of the Ricci curvature of $D$ and this tensor and its covariant derivative are totally isotropic, see \cite{hsstz-walker}.
We note that formula \eqref{ambient-local} for $\mb{g}$ says that the ambient metric is in fact linear in $\rho$ and  the iterative procedure determining the ambient metric stops after the first step. 

A geometric proof of vanishing Ricci curvature is presented in the next section. This is based on a combination of the well known Patterson--Walker and Thomas cone constructions, both of which we recall.

\section{Geometric construction of the ambient metric}
The association $D\squig g$ generalizes to a natural association from projective to conformal structures. 
Recall that two affine connections $D,D'$ on $N$ are called projectively related or equivalent if they have the same geodesics as unparameterized curves, which is the case if and only if there exists a $1$-form 
$\Ups\in\Om^1(N)$ with
\begin{align}\label{projchange}
  D'_X Y=D_XY+\Ups(X)Y+\Ups(Y)X
\end{align}
for all $X,Y\in\X(N)$.
It is sufficient to restrict ourselves to \emph{special} connections in a projective class, i.e., to those that preserve some volume form.
For projective structures it is useful to employ a suitably scaled \emph{projective density bundle} $\ce_+(1)$, defined as the special case of \emph{weight} $w=1$ of
 $\mc{E}(w):=\left(\wedge^n T N\right)^{-\frac{w}{n+1}}.$
Then a section $s:N\goesto \ce_+(1)$ corresponds to a choice of a special affine connection $D$ in the projective equivalence class $[D]$, and any $s'=e^f s$ corresponds to $D'$ projectively related to $D$ via \eqref{projchange} with $\Ups=df$.
We define $M=T^*N(2)$ the (projectively) weighted co-tangent bundle of $N$. Then, as was shown in \cite{hsstz-walker}, two projectively related affine connections $D,D'$ on $N$ induce two conformally related metrics $g,g'$ on
$M$, and we therefore have a natural association $(N,[D])\squig (M,[g])$. 

The cone $\mc{C}:=\ce_+(1)$ carries the canonical and well known Ricci-flat \emph{Thomas cone connection} $\na$, see \cite{thomas-invariants} or \cite{cap-slovak-book}, where the specific weight $1$ (different from \cite{cap-slovak-book}) is most convenient for our computations. We will need a local formula for $\na$: Let $s:N\goesto \ce_+(1)$ be the scale corresponding to an affine connection $D\in[D]$, providing a trivialization
$\ce_+(1)\overset{}{\cong} \rr_+ \times N$ via
 $(x^0,x)\mapsto s(x)x^0.$
In this trivialization the Thomas cone connection is given by
\begin{align}\label{thomas}
  \na_XY=D_X Y-\frac{1}{n-1}\Ric(X,Y)Z,\;
  \na Z=\id_{T\mc{C}}
\end{align}
where $X,Y\in\X(N)$ and $Z=x^0\pa_{x^0}$ is the Euler field on $\mc{C}$. It is in fact easy to see directly from formula \eqref{thomas} that the thus defined affine connection $\na$ on the Thomas cone $\mc{C}$ is independent of the choice of scale and Ricci-flat.

Employing a local coordinate patch on $N$ which induces coordinates $x^A,\ y_A$ on the co-tangent bundle $T^*N$ and coordinates $x^0,x^A,y_A,y_0$ on $T^*\mc{C}\cong \rr_+\times T^*N\times \rr$, the Patterson--Walker metric $\mb{g}$ associated to $\na$ is
\begin{align}
  \mb{g}=&2dx^A\odot dy_A+2dx^0\odot dy_0-\frac{4}{x^0}y_Bdx^0\odot dx^B\\ &-2y_C\Gamma \ind{_A^C_B} dx^A\odot dx^B+2\frac{x^0y_0}{n-1}\Ric_{AB}dx^A\odot dx^B. \notag
\end{align}
Ricci-flatness of $\mb{g}$ follows directly from Ricci-flatness of $\na$, see Theorem 2 of \cite{hsstz-walker}.
Via the change of coordinates
$t=x^0,\ \rho=\frac{y_0}{x^0},\ p_A=\frac{y_A}{(x^0)^{2}}$
the metric $\mb{g}$ transforms to \eqref{ambient-local}, which is the form of a Fefferman--Graham ambient metric \eqref{FG}. In particular this shows Theorem \ref{thm1}.

\begin{rema}
 As a Patterson--Walker metric $(\mb{M},\mb{g})$ carries a naturally induced homothety $\mb{k}$ of degree $2$, which takes the form $2p_A \pa_{p_A}+2\rho\pa_{\rho}$. According to Lemma 5.1 of \cite{hsstz-walker} the infinitesimal affine symmetry $Z$ of $\na$ lifts to a Killing field, which one computes as $t\pa_t-2p_A \pa_{p_A}-2\rho\pa_{\rho}$. In particular it follows that the Euler field $t\pa_t$ of the Fefferman--Graham ambient metric $\mb{g}$ can be written as the sum of this Killing field and the homothety $\mb{k}$.
$T\mb{M}$ carries the maximally isotropic $(n+1)$-dimensional subspace 
spanned by $\{\pa_{p_A}, \pa_{\rho}\}$ which is  preserved by $\bna$. This subspace can be equivalently described by a $\bna$-parallel pure spinor $\mb{s}$ on $\mb{M}$. 
The ambient Killing field $\mb{k}$ and the ambient parallel pure spinor $\mb{s}$ correspond to a homothety $k$ of $g$ and a parallel pure spinor $\chi$ on $M$ that belong to the characterizing objects of the Patterson--Walker metric $g$, see Theorem 1 of \cite{hsstz-walker}.
\end{rema}

We conclude this section by summarizing the construction:
\begin{thm}\label{thm2}
Given a projective structure  $(N,[D])$ on an $n$-dimensional manifold $N$, 
the geometric constructions indicated in the following diagram commute:
\begin{align*}
  \xymatrix{
  (\mc{C},\na)   \ar@{~>}^{}[r]&  (\mb{M},{\mb{g}})\\
  (N,[D]) \ar@{~>}^{}[u] \ar@{~>}^{}[r]& (M,[g]) \ar@{~>}^{}[u] 
  }
\end{align*}
In particular, the induced conformal structure $[g]$ admits a globally Ricci-flat Fefferman--Graham ambient metric $\mb{g}$ which is itself a Patterson--Walker metric.
\end{thm}

 We remark here that, for generic $[D]$, the resulting conformal class $[g]$ does not contain an Einstein metric, see Theorem 2 of \cite{hsstz-walker}. In particular, one obtains a large class of conformal structures which are not conformally Einstein but which admit globally Ricci-flat and explicit ambient metrics.

\section{Vanishing $Q$--curvature}
The $Q$-curvature $Q_g$ of a given metric $g$ is a Riemannian scalar invariant with a particularly simple transformation law with respect to conformal change of metric. It has been introduced by T. Branson in \cite{branson-functional} and has been the subject of intense research in recent years, see e.g. \cite{chang-eastwood-Q} for an overview. Computation of $Q$-curvature is notoriously difficult, see e.g. \cite{gover-peterson}. An explicit form of a Fefferman--Graham ambient metric $\mb{g}$ for a given metric $g$ allows a computation of $Q_g$. Using the fact that $\mb{g}$ is actually a Patterson--Walker metric, this computation is particularly simple.
\begin{thm}
The Patterson--Walker metric $g$ associated to a volume--preserving, torsion--free affine connection $D$ has vanishing $Q$-curvature $Q_g$.
\end{thm}
\begin{proof}
We follow the computation method for $Q_{g}$ from \cite{fefferman-hirachi}: For this it is necessary
to compute $-\blp^{n}\log(t)$,
where $\blp$ is the ambient Laplacian on $\mb{M}=\rr_+\times T^*N\times\rr$ associated to $\mb{g}$ and $t:\mb{M}\goesto\rr_+$ is the first coordinate. Restricting $-\blp^n\log(t)$
to the cone $\rr_+\times T^*N\times\{0\}$ and evaluating at $t=1$ yields $Q_{g}$ (see \cite{fefferman-hirachi}). To show that Q-curvature vanishes for $g$, it is in particular sufficient to show that $\blp \log(t)=0$. However, the function $t:\mb{M}\goesto\rr_+$ is horizontal since it is just the pullback of the coordinate function $x^0:\ce_{+}(1)\goesto\rr_+$ on the Thomas cone $\mc{C}\cong \rr_+\times N$ via the canonical projection $T^*\mc{C}\overset{}\goesto\mc{C}$.
It follows from the explicit formula for the Christoffel symbols 
of a Patterson--Walker metric that $\blp$ vanishes on any horizontal function, see \cite{hsstz-walker}, section 2.1. Thus in particular $\blp \log(t)=0$, and then also $Q_g=0$.
\end{proof}

\newcommand{\etalchar}[1]{$^{#1}$}
\def\polhk#1{\setbox0=\hbox{#1}{\ooalign{\hidewidth
  \lower1.5ex\hbox{`}\hidewidth\crcr\unhbox0}}}

\end{document}